\documentclass{amsart}
\usepackage{amssymb,amsthm}

\newtheorem{theorem}{Theorem}
\newtheorem{lemma}[theorem]{Lemma}
\newtheorem{remark}[theorem]{Remark}
\newtheorem{conjecture}[theorem]{Conjecture}
\theoremstyle{definition}
\newtheorem{definition}[theorem]{Definition}

\newcommand{\F}{{\mathcal{F}}}

\newcommand{\C}{\mathrm C}
\newcommand{\D}{\mathrm D}
\newcommand{\K}{\mathrm K}
\newcommand{\V}{\mathrm V}
\newcommand{\W}{\mathrm W}
\newcommand{\ZZ}{\mathbb Z}

\newcommand{\Aut}{\mathrm{Aut}}
\newcommand{\Alt}{\mathrm{Alt}}
\newcommand{\Sym}{\mathrm{Sym}}
\newcommand{\Cay}{\mathrm{Cay}}
\newcommand{\Cos}{\mathrm{Cos}}

\newcommand{\la}{\langle}
\newcommand{\ra}{\rangle}

\renewcommand{\wr}{\mathop{\rm wr}}

\begin{document}
\title{Tetravalent arc-transitive graphs with unbounded vertex-stabilisers}

\author[P. Poto\v{c}nik]{Primo\v{z} Poto\v{c}nik}
\address{Primo\v{z} Poto\v{c}nik,\newline
 Faculty of Mathematics and Physics,
 University of Ljubljana \newline 
Jadranska 19, 1000 Ljubljana, Slovenia}\email{primoz.potocnik@fmf.uni-lj.si}

\author[P. Spiga]{Pablo Spiga}
\address{Pablo Spiga,\newline
  School of Mathematics and Statistics,
The University of Western Australia,\newline
 Crawley, WA 6009, Australia} \email{spiga@maths.uwa.edu.au}

\author[G. Verret]{Gabriel Verret}
\address{Gabriel Verret,\newline
 Institute of Mathematics, Physics, and
  Mechanics, \newline 
Jadranska 19, 1000 Ljubljana, Slovenia}
\email{gabriel.verret@fmf.uni-lj.si}

\thanks{Address correspondence to P. Spiga; 
E-mail: spiga@maths.uwa.edu.au\\ 
The second author is supported by UWA as part of the
Australian Council Federation Fellowship Project FF0776186.}


\begin{abstract}
It has long been known that there exist finite connected tetravalent arc-transitive graphs with arbitrarily large vertex-stabilisers.
However, beside a well known family of exceptional graphs, related to the lexicographic product of a cycle with an edgeless graph on two vertices, only a few such infinite families of graphs are known. In this paper, we present two more families of tetravalent arc-transitive graphs with large vertex-stabilisers, each significant for its own reason.
\end{abstract}

\maketitle

\section{Introduction}
\label{intro}
A celebrated theorem of Tutte \cite{Tutte,Tutte2} states that, in a finite connected cubic arc-transitive graph, a vertex-stabiliser has order at most $48$. As is well known and will be shown below, Tutte's result has no immediate generalization to graphs of valency $4$. However, it is still an interesting question whether there exists a relatively tame function $f$, such that, for every connected tetravalent arc-transitive graph with $n$ vertices, a vertex-stabiliser has order at most $f(n)$.

Here is a standard example showing that $f$ must grow at least exponentially with the number of vertices: for $r\ge 3$, let $\W_r$ denote the lexicographic product $\C_r[\bar{\K}_2]$ of a cycle of length $r$ with an edgeless graph on two vertices. The graph $\W_r$ has vertex set $\ZZ_r\times \ZZ_2$ with $(v,i)$ adjacent to  $(v\pm 1, j)$ for $v\in \ZZ_r$ and $i,j\in \ZZ_2$. This graph admits an arc-transitive action of the group $G \cong \C_2\wr \D_{r}$, with the base group $\C_2^r \le G$ preserving each \textit{fibre} $V_j = \{j\} \times \ZZ_2 \subseteq \ZZ_r\times \ZZ_2$ setwise and $\D_r$ acting naturally on the set of fibres $\{V_j\}_j$. The vertex-stabiliser $G_v$ is then isomorphic to the group $\C_2^{r-1}\rtimes \C_{2}$. In particular, the order of $G_v$ is $2^r$ and grows exponentially with the number of vertices $2r$ of the graph $W_r$. Further examples of families of tetravalent arc-transitive graphs exhibiting an exponential growth of $|G_v|$ were found by Praeger 
 and Xu in \cite{PraegerXu}. The graphs $\C(r,s)$, constituting these families,  will be described in Section~\ref{crs}.

On the other hand, it has been shown recently by the authors of this paper that a completely different picture emerges once the exceptional graphs $\C(r,s)$ are excluded~\cite[Corollary 3]{main}. Namely, if $\Gamma$ is a connected tetravalent $G$-arc-transitive graph not isomorphic to a graph $\C(r,s)$, then either $|G_v| \le 2^43^6$ or
\begin{equation}
\label{bound}
|\V\Gamma| \ge 2 |G_v| \log_2(|G_v|/2).
\end{equation}
Note that \eqref{bound} implies that $|G_v|$ is bounded above by a sub-linear function of $|\V\Gamma|$. The first aim of this paper is to construct a family of connected $G$-arc-transitive graphs, not isomorphic to $\C(r,s)$, attaining the bound given in \eqref{bound}. This will done in Section~\ref{Gamma} (see Definition ~\ref{defGamma}). Both the graphs $\C(r,s)$ and the graphs presented in Section~\ref{Gamma} have soluble groups of automorphisms. 

The only family of connected tetravalent $G$-arc-transitive graphs with arbitrarily large vertex-stabilisers and 
with $G$ non-soluble that was previously known to us is the family constructed by Conder and Walker in~\cite{CW}. For a member $\Gamma$ of this family, we have that $G\cong \Sym(n)$ (for some $n$) and  $|\V\Gamma| \ge (|G_v|-1)!$. In particular, $|G_v|$ grows slower than any logarithmic function of $|\V\Gamma|$.

The second family of graphs that we construct in this paper (see Definition~\ref{def:D}) consists of tetravalent $G$-arc-transitive graphs $\Delta_n$ with $G\cong \Sym(4n)$.
The vertex-stabiliser $G_v$ in this family has order $2^{2n}$, and hence $|\V\Delta_n| = (4n)! /2^{2n}$.
Using Stirling's formula, one can see that, asymptotically,  $|G_v|$ grows slower than $|\V\Gamma|^c$ for any $c>0$, but faster than any logarithmic function of $|\V\Gamma|$.
This is much slower than the growth from \eqref{bound} but still considerably faster than the growth exhibited by the family in~\cite{CW}.

\section{Preliminaries}

The following standard notation and terminology will be used throughout the paper. All the graphs will be finite, simple and connected.
Let $\Gamma$ be a graph and let $G\le \Aut(\Gamma)$. We say that $\Gamma$ is $G$-arc-transitive provided that $G$ acts transitively on the set of arcs of $\Gamma$. In this case, the permutation group $G_v^{\Gamma(v)}$ induced by the action of the stabiliser $G_v$ of a vertex $v\in \V\Gamma$ on the neighbourhood $\Gamma(v)$ is transitive. A pair $(\Gamma,G)$ is called {\em locally-$\D_4$} if $\Gamma$ is a connected tetravalent $G$-arc-transitive graph with $G_v^{\Gamma(v)}$ isomorphic to the dihedral group  $\D_4$ in its action on $4$ points. For a group $G$, a subgroup $H$ and an element $a\in G\setminus H$, the {\em coset graph} $\Cos(G,H,a)$ is the graph with vertex set the set of right cosets $G/H = \{Hg : g \in G\}$ and edge set $\{ \{Hg, Hag\} : g \in G\} $. It was proved by Sabidussi \cite{Sabidussi} that every $G$-vertex-transitive graph is isomorphic to some coset graph of $G$. More precisely, we have the following well known result.

\begin{lemma}
\label{lem:cos}
Let $\Gamma$ be a connected $G$-arc-transitive graph, let $G_v$ be the stabiliser of the vertex $v\in \V\Gamma$ and let  $a\in G$ be an automorphism with $v^a \in \Gamma(v)$.
Then $\Gamma \cong \Cos(G,G_v,a)$.

Conversely, 
let $H$ be a core-free subgroup of a finite group $G$ and let $a\in G$ be such that $G=\langle H,a\rangle$ and $a^{-1} \in HaH$.
Then the graph $\Gamma=\Cos(G,H,a)$ is connected and $G$-arc-transitive. The valency of $\Gamma$ is $|HaH|/|H|$ and the neighbourhood $\Gamma(H)$ of
the vertex $H\in \V\Gamma$ is the set $\Gamma(H)=\{Hah : h\in H\}$.
\end{lemma}

Given a connected graph $\Gamma$, a subgroup $G\le \Aut(\Gamma)$ and a normal subgroup $N\unlhd G$, the {\em normal quotient graph} $\Gamma/N$ is the graph with vertex set the orbit space $\V\Gamma/N = \{ v^N : v \in \V\Gamma\}$ and with two orbits $u^N$ and $v^N$ adjacent in $\Gamma/N$ whenever there exists a pair of vertices $u'$, $v'\in \V\Gamma$ with $u'\in u^N$ and $v'\in v^N$. Note that there exists a natural (but possibly not faithful) action of $G/N$ on $\Gamma/N$. Moreover, if $G$ is transitive on the vertices (arcs, respectively) of $\Gamma$, then $G/N$ is transitive on vertices (arcs, respectively) of $\Gamma/N$.

A special case of normal quotients arises when the quotient projection $\pi \colon \Gamma \to \Gamma/N$, $v \mapsto v^N$, is locally bijective (that is, $\pi$ maps the neighbourhood of an arbitrary vertex $v\in \V\Gamma$ bijectively onto the neighbourhood of $\pi(v)$ in $\Gamma/N$). It is well known that, in this case, $G/N$ acts faithfully on $\V(\Gamma/N)$, and that the vertex-stabilisers  $G_v$ and $(G/N)_{\pi(v)}$ are isomorphic and induce permutation isomorphic local groups $G_v^{\Gamma(v)}$ and $(G/N)_{\pi(v)}^{\Gamma/N(\pi(v))}$.
When the quotient projection $\pi \colon \Gamma \to \Gamma/N$ is locally bijective, we will say that the pair $(\Gamma,G)$ is an {\em $N$-cover} of the pair $(\Gamma/N,G/N)$. If, in addition, $N$ is contained in the centre of $G$, then we say that the pair $(\Gamma,G)$ is a {\em central $N$-cover} of the pair $(\Gamma/N,G/N)$.

We will need the following lemma describing the relationship between covers and coset graphs.

\begin{lemma}
\label{lem:CosCov}
Let $G$ be a group generated by a core-free subgroup $H$ and an element $a$. Further, let $\Gamma = \Cos(G,H,a)$ 
 and let $N$ be a normal subgroup of $G$ not containing $a$ and intersecting the set $H^aH$ trivially. Let $\bar{G} = G/N$, let $\bar{H} = HN/N$, and let $\bar{a} = Na \in G/N$. Then $\Gamma/N \cong \Cos(\bar{G},\bar{H},\bar{a})$ and $(\Gamma,G)$ is an $N$-cover of $(\Gamma/N,\bar{G})$.
\end{lemma}

\begin{proof}
Let $v$ denote the vertex of $\Gamma=\Cos(G,H,a)$ corresponding to the coset $H\in G/H$. Then $H=G_v$ and $H^a = G_{v^a}$. As $\Gamma$ is $G$-arc-transitive, to show that the quotient projection $\pi\colon\Gamma \to \Gamma/N$ is locally bijective, it suffices to show that the $N$-orbit of $v^a$ intersects the neighbourhood of $v$ only in $v^a$, that is, $(v^a)^N \cap (v^a)^H = \{v^a\}$.
Now, if $u\in (v^a)^N \cap (v^a)^H$, then $u=v^{az} = v^{ah}$ for some $z\in N$ and $h\in H$. Therefore $azh^{-1}a^{-1} \in G_v = H$, and thus $z\in H^aH$. Since $N\cap H^aH  = 1$, this implies that $z=1$ and that $u=v^a =v^{ah}$, and hence $(v^a)^N \cap (v^a)^H = \{v^a\}$. This shows that $\pi \colon \Gamma \to \Gamma/N$ is indeed a covering projection and hence $(\Gamma,G)$ is an $N$-cover of $(\Gamma/N,G/N)$. It follows that $G/N$ acts faithfully and arc-transitively on $\Gamma/N$. The stabiliser of the vertex $v^N$ in $G/N$ is the group $\bar{H} = HN/N \cong H/(H\cap N) \cong H$, and $\bar{a}$ maps the vertex $\pi(v)$ to the neighbour  $\pi(v^a)$. By Lemma~\ref{lem:cos} we may thus conclude that $\Gamma/N \cong \Cos(\bar{G},\bar{H},\bar{a})$, as claimed.
\end{proof}

\section{The family of graphs with exponential growth of the vertex-stabiliser}
\label{crs}

In this section we describe the family of graphs $\C(r,s)$ mentioned in Section~\ref{intro}, which generalise the graphs $\W_r$.  We give a definition which is slightly different, but equivalent to the definition used in~\cite{PraegerXu}, where they were first introduced.

Let $\C(r,1) = \W_r$. Let $s$ be an integer satisfying $2\le s \le r-2$ and let $\C(r,s)$ be the graph with vertices being the $(s-1)$-paths of $\C(r,1)$ containing at most one vertex from each fibre $V_j$, $j\in\ZZ_r$, and with two such $(s-1)$-paths being adjacent in $\C(r,s)$ if and only if their intersection is an $(s-2)$-path in $\C(r,1)$.
The number of vertices of $\C(r,s)$ is clearly
\begin{equation}
|\V\C(r,s)| =r 2^s.
\end{equation}

It is easy to see that the girth of $\C(r,s)$ is $4$. Further, $\C(r,s)$ is bipartite provided that $r$ is even.

For $i\in \ZZ_r$, let $x_i$ denote the automorphism of $\W_r$ which interchanges the two vertices in the fibre $V_i$ and fixes all other vertices. Further, let $a$ be the automorphism of $\W_r$ which maps each $(v,i)\in\V\W_r$ to $(v+1,i)$, and let $b$ be the automorphism acting on the vertices of $\W_r$ according to the rule $(v,i)^b$ = $(-v,i)$ for every $v\in \ZZ_r$ and $i\in \ZZ_2$. Then the group
\begin{equation}
G_r = \langle x_0, \ldots, x_{r-1} \rangle \rtimes \langle a, b\rangle \cong \C_2^r \rtimes \D_r
\end{equation}
acts arc-transitively on $\W_r$. It was shown in \cite{PraegerXu} that, if $r\not = 4$, then $\Aut(\W_r) = G_r$  ($\W_4\cong K_{4,4}$ and hence $\Aut(\W_4) \cong \Sym(4) \wr \Sym(2)$).

Since $G_r$ permutes the $(s-1)$-paths of $W_r$ containing at most one vertex from each fibre $V_j$, the group $G_r$ acts as a group of automorphisms of $\C(r,s)$.
The stabiliser in $G=G_r$ of the vertex $v$ of $\C(r,s)$ corresponding to the 
$(s-1)$-path $(r-s,0)(r-s+1,0)\ldots(r-1,0)$ in $\W_r$
is the group $H = \langle x_0, x_1, \ldots, x_{r-s-1}, b_s\rangle$, where $b_s$ is the element of $G$ acting as $(v,i)^{b_s} = (r-s-1-v,i)$.
Note that $G = \langle H , a \rangle$. Using Lemma~\ref{lem:cos}, it is now easy to see that the graphs $\C(r,s)$ can be defined in terms of coset graphs as follows.

\begin{lemma}
\label{lem:CrsCos}
The graph $\C(r,s)$ is isomorphic to the coset graph $\Cos(G_{r,s},H_{r,s},a)$ where
\begin{eqnarray*}
G_{r,s} & = & \langle x_0, \ldots, x_{r-1},a,b \mid x_0^2= \cdots = x_{r-1}^2 =a^r = b^2 = (ab)^2=1,\\ & & \phantom{\langle x_0, \ldots, x_{r-1},a,b \mid i} x_i^a =x_{i+1},\> x_i^b = x_{r-s-1-i} \rangle, \\
H_{r,s} & = & \langle x_0, \ldots, x_{r-1}, b \rangle \le G_{r,s}.
\end{eqnarray*}  
\end{lemma}

Let us finish this section by reporting the following result from \cite{PraegerXu} regarding the automorphism group of $\C(r,s)$.

\begin{lemma}
{\rm \cite[Lemma 2.12]{PraegerXu}}
\label{lem:AutCrs}
Let $\Gamma=\C(r,s)$ with $1\le s \le r-1$. If $r\not =4$, then $\Aut(\Gamma) = G_{r,s}$. Moreover, $\Aut(\C(4,1)) \cong \Aut(\hbox{K}_{4,4}) \cong \Sym(4) \wr \Sym(2)$, 
 $\Aut(\C(4,2)) \cong \Sym(2) \wr \Sym(4)$ and $\Aut(\C(4,3)) \cong (\Sym(2) \wr \D_4).\Sym(2)$.
\end{lemma}

\section{The graphs attaining the bound \eqref{bound}} 
\label{Gamma}

In this section, for every $t\geq 2$, we construct two locally-$\D_4$ pairs $(\Gamma_t^{+},G_t^+)$ and $(\Gamma_t^{-},G_t^-)$ with $|\V\Gamma_t^\pm| = t2^{t+2}$ and $|G_t^\pm| = t2^{2t+3}$.
Since $|(G_t^\pm)_v| = |G_t^\pm| /  |\V\Gamma_t^\pm|= 2^{t+1}$, we see that the pairs $(\Gamma_t^{\pm},G_t^\pm)$ indeed meet the bound \eqref{bound} stated in Section~\ref{intro}.

Let $t$ be an integer satisfying $t\ge 2$. We start by considering the {\em extraspecial group} $E_t$ of order $2^{2t+1}$ of {\em plus type}, which has the following presentation:
\begin{eqnarray}\label{eq:E}\nonumber
E_t=\langle x_0,\ldots,x_{2t-1},z&\mid&
x_i^2=z^2=[x_i,z]=1 \textrm{ for }0\leq i\leq 2t-1,\\
&&[x_i,x_j]=1 \textrm{ for }|i-j|\neq t,\\\nonumber
&&[x_i,x_{t+i}]=z \textrm{ for  }0\leq i\leq t-1\rangle.
\end{eqnarray}

We will now extend the group $E_t$ by the dihedral group
\begin{eqnarray}\label{eq:D}
\D_{2t}=\langle a,b\mid a^{2t}=b^2=1, a^b=a^{-1}\rangle,
\end{eqnarray}
using two different 2-cocycles.
In both extensions, the generators $a$ and $b$ will act upon the generators of $E_t$ according to the rules:
$$
 x_i^a=x_{i+1} \hbox{ and } x_i^b=x_{t-1-i}\> \hbox{  for } 0\leq i\leq 2t-1\> \hbox{ (with indices taken mod } 2t).
$$
To obtain the split extension $G_t^+$, we let $a^{2t} = b^2 =1$, and thus define:
\begin{equation}
\label{G+}
G_t^+ = E_t \rtimes \D_{2t}, \quad  a^{2t}=b^2= 1,\> x_i^a=x_{i+1},\> x_i^b=x_{t-1-i}.
\end{equation}
The second extension $G_t^-$ is non-split, we have $a^{2t}=z$ and $b^2 =1$:
\begin{equation}
\label{G-}
G_t^- = E_t.\D_{2t}, \quad a^{2t}=z,\> b^2 =1,\> x_i^a=x_{i+1},\> x_i^b=x_{t-1-i}.
\end{equation}
Finally, let 
\begin{equation}
\label{H}
H_t^\pm = \langle x_0,\ldots,x_{t-1}, b\rangle \le G_t^\pm,
\end{equation}
and observe that $H_t^+ \cong H_t^- \cong \C_2^t \rtimes \C_2$.
The graphs $\Gamma_t^+$ and $\Gamma_t^-$ are now defined as coset graphs on the groups $G_t^+$ and $G_t^-$, respectively:
\begin{definition}
\label{defGamma}
Let $t\geq 2$ and let $a$, $H_t^{\pm}$ and $G_t^{\pm}$ be as above. Then $\Gamma_t^+ = \Cos(G_t^+,H_t^+,a)$ and $\Gamma_t^- = \Cos(G_t^{-},H_t^-,a)$.
\end{definition}

Before stating the main theorem of this section, let us first show that for any $t\ge 2$, a triple $(G,H,a) = (G_t^\pm,H_t^\pm,a)$ 
satisfies the conditions of Lemma~\ref{lem:cos} and thus gives rise to a connected $G$-arc-transitive graph $\Gamma_t^\pm$.
Observe first that since $a$ cyclically permutes the elements of the generating set $\{x_0, \ldots,x_{2t-1}\}$ of $E$, and since $x_0\in H$,
the group $\langle H, a\rangle$ contains the subgroup $E$. Since $b\in H$, we see that $\langle H,a \rangle = \langle H, b, a\rangle = \langle E,b,a\rangle = G$. 
The graph $\Gamma_t^\pm$ is therefore connected.

To see that $H$ is core-free in $G$ observe that 
$$
 H \cap H^{a^t}=\langle x_0,\ldots,x_{t-1}, b \rangle \cap \langle x_t,\ldots,x_{2t-1}, b\rangle = \langle b \rangle.
$$
Since $b\notin H^a$, we see that
$H\cap H^{a^t}\cap H^a=1$,  implying that the core of $H$ in $G$ is trivial.
 Finally,  since $b\in H$, it follows that $HaH = HbabH = H a^{-1} H$, and hence $a^{-1} \in H a H$.
 In view of Lemma~\ref{lem:cos}, this implies that the graph $\Gamma_t^\pm$ is indeed connected and $G$-arc-transitive.

\begin{theorem}
\label{lemma:construction}
Let $t$ be an integer with $t\ge 2$ and let $(\Gamma,G)$ be either $(\Gamma_t^+,G_t^+)$ or $(\Gamma_t^-,G_t^-)$. 
Then the following statements hold:
\begin{itemize}
 \item[{\rm (i)}]
 $|\V\Gamma| = t2^{t+2}$ and $|G_v| = 2^{t+1}$, and hence the bound \eqref{bound} in Section~\ref{intro} is met;
\item[{\rm (ii)}]
 $(\Gamma,G)$ is a central $\C_2$-cover of the pair $(\C(2t,t),G_{2t,t})$, where $G_{2t,t}$ is as in Lemma~\ref{lem:CrsCos};
\item[{\rm (iii)}]
  $(\Gamma,G)$ is a locally-$\D_4$ pair;
 \item[{\rm (iv)}]
 the girth of $\Gamma$ is $4$ if $\Gamma=\Gamma_2^+$, is $6$ if $\Gamma=\Gamma_3^+$, and is $8$ otherwise;
 \item[{\rm (v)}]
 if $(\Gamma,G) \not = (\Gamma_2^-, G_2^-)$, then $\Aut(\Gamma) = G$, while $|\Aut(\Gamma_2^-):G_2^-|=9$.
\end{itemize}
\end{theorem}

\begin{proof}
Let $E=E_t$ and let $H=H_t^+$ or $H_t^-$, so that $\Gamma=\Cay(G,H,a)$.
To prove part (i), note that $|G|  = |E_t||\D_{2t}|=2^{2t+1}4t=t2^{2t+3}$ and that $H \cong \langle x_0,\ldots, x_{t-1}\rangle \rtimes \langle b\rangle \cong C_2^t \rtimes \C_2$, implying that $|G_v|=|H| = 2^{t+1}$.
Therefore,  $|\V\Gamma|  = |G|/|H| =t2^{t+2}$, as claimed.

To prove part (ii), first observe that $\langle z \rangle \cong \C_2$ is contained in the centre of $G$, that $G/\langle z\rangle \cong G_{2t,t}$ and that the natural isomorphism between $G/\langle z \rangle $ and $G_{2t,t}$ maps the group $H$ bijectively onto the group $H_{2t,t} \le G_{2t,t}$ (defined in Lemma~\ref{lem:CrsCos}), and the element $a\in G$ from the definition of the graph $\Gamma=\Cos(G,H,a)$ to the element $a$ from the definition of the graph $\C(2t,t) = \Cos(G_{2t,t}, H_{2t,t},a)$. 

We now show that $z\not \in H^aH$. Suppose, by contradiction, that $z\in H^aH$. Every element of $H$ is of the form $eb^\epsilon$ for some $e\in \langle x_0,\ldots ,x_{t-1}\rangle\leq E$ and $\epsilon\in\{0,1\}$, hence $z$ can be written in the form $(eb^{\epsilon})^a(e'b^{\epsilon'})$. Since $E$ is normal in $G$, we have $z=c(b^{\epsilon})^ab^{\epsilon'}$, for some $c\in E$. Since $z\in E$, it follows that $(b^{\epsilon})^ab^{\epsilon'}=1$ and hence $\epsilon=\epsilon'=0$. It follows that $z$ can be written in the form $(x_0^{\epsilon_0}\ldots x_{t-1}^{\epsilon_{t-1}})^a(x_0^{\epsilon'_0}\ldots x_{t-1}^{\epsilon'_{t-1}})=(x_1^{\epsilon_0}\ldots x_{t}^{\epsilon_{t-1}})(x_0^{\epsilon'_0}\ldots x_{t-1}^{\epsilon'_{t-1}})$.
If $\epsilon_{t-1} = 0$, then the latter belongs to the elementary abelian group $\la x_0, \ldots, x_{t-1}\ra$, which does not contain $z$. Similarly, if $\epsilon'_0 = 0$, then the latter belongs to the elementary abelian group $\la x_1,\ldots,x_t\ra$, which does not contain $z$.
Hence we may assume that $\epsilon_{t-1} = \epsilon'_0=1$. Now, since $x_tx_0 = x_0x_tz$, it follows that 
\begin{eqnarray*}
z&=& x_1^{\epsilon_0}\cdots x_{t-1}^{\epsilon_{t-2}}x_tx_0x_1^{\epsilon'_1}\cdots x_{t-1}^{\epsilon'_{t-1}}\\
&=&x_1^{\epsilon_0}\cdots x_{t-1}^{\epsilon_{t-2}}x_0x_tzx_1^{\epsilon'_1}\cdots x_{t-1}^{\epsilon'_{t-1}}\\
&=&x_1^{\epsilon_0}\cdots x_{t-1}^{\epsilon_{t-2}}x_0x_1^{\epsilon'_1}\cdots x_{t-1}^{\epsilon'_{t-1}}x_tz=d x_t z\\
\end{eqnarray*} with $d\in \la x_0, \ldots x_{t-1}\ra$. Therefore $x_t \in \la x_0, \ldots, x_{t-1}\ra$, which is
clearly a contradiction. Thus we have  shown that $z\not\in H^aH$.
It follows by  Lemma~\ref{lem:CosCov} that $\Gamma/\la z \ra \cong \C(2t,t)$ and that $\Gamma$ is a $\la z \ra$-cover of $\C(2t,t)$. Part (ii) of the theorem is thus proved.

Moreover, since the pair $(\C(2t,t), G_{2t,t})$ is locally-$\D_4$, so is the covering pair $(\Gamma, G)$, thus proving part (iii).

In the proof of parts (iv) and (v) we will need detailed information about the spheres of radius $2$ and $3$ around the vertex $H$. For $i\geq 1$ and $v\in \V\Gamma$, let $\Gamma_i(v)$
denote the set of vertices in $\Gamma$ at distance $i$ from $v$. To determine the neighbourhood $\Gamma_1(H)$, observe that $Ha, Ha^{-1}=Hab$, $Hx_{2t-1}a = Hax_0$ and $Hx_ta^{-1} = Ha^{-1}x_{t-1}=Habx_{t-1}$
are four pairwise distinct cosets of the form $Hah$ with $h\in H$. Since $\Gamma$ has valency $4$, this implies that
\begin{equation}
\label{GammaH}
\Gamma_1(H) = \{Hg : g \in X_1\},\> \hbox { where } X_1 = \{ a, x_{2t-1}a, a^{-1}, x_ta^{-1}\}.
\end{equation}   
Further, observe that 
\begin{equation}
\Gamma_2(H) =  \{ Hg : g\in X_2\} \setminus \{H\},\> \hbox{  where } X_2 = \{gh : g,h \in X_1\}.
\end{equation}
An easy computation shows that
\begin{eqnarray}
\label{Gamma2H}
 X_2 & = & \{x_{2t-2}^{\epsilon_1}x_{2t-1}^{\epsilon_2}a^2 : \epsilon_1,\epsilon_2 \in \{0,1\}\} \> \cup \\
 \nonumber    & & \{x_{t}^{\epsilon_1}x_{t+1}^{\epsilon_2}a^{-2} : \epsilon_1,\epsilon_2 \in \{0,1\}\} \> \cup \\
  \nonumber      &    & \{x_t,\, x_tz,\, x_{2t-1},\, x_{2t-1}z \}.
\end{eqnarray}
Similarly, note that 
\begin{equation}
\Gamma_3(H)  =  \{ Hg : g\in X_3\} \setminus \Gamma_1(H),\> \hbox{  where } X_3 = \{gh : g,h \in X_1, X_2\}.
\end{equation}
By a straightforward computation we get
\begin{eqnarray}
\label{Gamma3H}
 X_3 & = & \{x_{2t-3}^{\epsilon_1}x_{2t-2}^{\epsilon_2}x_{2t-1}^{\epsilon_3}a^3 : \epsilon_1,\epsilon_2,\epsilon_3 \in \{0,1\}\} \> \cup \\
 \nonumber  & & \{x_{t}^{\epsilon_1}x_{t+1}^{\epsilon_2}x_{t+2}^{\epsilon_3}a^{-3}: \epsilon_1,\epsilon_2,\epsilon_3 \in \{0,1\} \}\ \> \cup \\
\nonumber                        & & \{x_{t}^{\epsilon_1}x_{2t-1}^{\epsilon_2}z^{\epsilon_3}a: \epsilon_1,\epsilon_2,\epsilon_3 \in \{0,1\} \}  \> \cup  \\
 \nonumber                       &  & \{x_{2t-2}x_{2t-1}^{\epsilon_1}z^{\epsilon_2}a : \epsilon_1, \epsilon_2\in\{0,1\} \}  \> \cup  \\
\nonumber   & & \{x_{t}^{\epsilon_1}x_{t+1}^{\epsilon_2}z^{\epsilon_3}a^{-1}: \epsilon_1,\epsilon_2,\epsilon_3 \in \{0,1\} \}  \> \cup \\
 \nonumber                       &  & \{x_{t}^{\epsilon_1} x_{2t-1}z^{\epsilon_2}a : \epsilon_1, \epsilon_2\in\{0,1\} \}.
\end{eqnarray}

Having computed the second and the third neighbourhood of the vertex $H$, it is now easy to determine the girth of the graph $\Gamma$. Recall first that $\Gamma$ is a $2$-fold cover of the
graph $\C(2t,t)$, which is bipartite and of girth $4$. This implies that $\Gamma$ is also bipartite and of girth not exceeding $8$. 

Now, if $t=2$ and $\Gamma = \Gamma_t^+$, then the order of $a$ is $4$, and hence the elements $a^{-2}, a^2$, listed in \eqref{Gamma2H} coincide. In particular, $\Gamma$ contains the $4$-cycle $(H,Ha,Ha^2, Ha^{-1})$, and thus the girth of $\Gamma$ is $4$.

In all the other cases (that is, if $\Gamma=\Gamma_2^-$ or if $t\ge 3$), the $12$ elements of $X_2$ listed in \eqref{Gamma2H} are a transversal of $12$ pairwise distinct $H$-cosets. This implies that the girth of $\Gamma$ is at least $6$. If $t=3$ and $\Gamma=\Gamma_t^+$, then $a^3=a^{-3}$ and therefore $\Gamma$ contains the $6$-cycle $(H,Ha,H^2,Ha^3, Ha^{-2}, Ha^{-1})$.
In particular, the girth of $\Gamma_3^+$ is $6$.

In all the other cases (that is, if $\Gamma=\Gamma_3^-$ or if $t\ge 4$),
the $36$ elements of $X_3$, listed in \eqref{Gamma3H}, are pairwise distinct, and in fact are a transversal of 36 pairwise
distinct $H$-cosets. The girth of $\Gamma$ is thus at least (and therefore exactly) $8$.
This proves part (iv) of the theorem.

Let us now prove part (v). The automorphism groups of $\Gamma_t^\pm$ for $t\le 3$ 
can be checked easily with \texttt{Magma} \cite{magma}. We will therefore assume that $t\ge 4$ and
set $\Gamma=\Gamma_t^\pm$ and $A=\Aut(\Gamma)$. 

We will first show that the orbits of $\langle z\rangle$ on $\V\Gamma$ form a system of imprimitivity for the action of $A$ on $\V\Gamma$. Set
\begin{equation}
\label{eq:Bdef}
 \mathcal{C} =  \bigcap_{u\in \Gamma(H)} \Gamma_3(u)\> =\> \Gamma_3(H)a\> \cap \> \Gamma_3(H)x_{2t-1}a  \> \cap  \>   \Gamma_3(H)a^{-1}  \> \cap  \> \Gamma_3(H)x_ta^{-1}
\end{equation}
and set $\mathcal{B}=\mathcal{C}\cup \{H\}$.
Using \eqref{Gamma3H}, a straightforward calculation shows that
\begin{equation}
\label{eq:B}
\mathcal{B}= \{H, Hx_tx_{2t-1},Hz,Hx_{t}x_{2t-1}z\}.
\end{equation}

 We claim that $\mathcal{B}$ is a block of imprimitivity for $A$. 
 Observe first that the
 setwise stabiliser $A_\mathcal{B}$ of the set $\mathcal{B}$ in $A$ acts transitively on $\mathcal{B}$ and contains the vertex stabiliser $A_H$.
 The latter follows directly from the definition of the set $\mathcal{B}$, while the former follows from the observation that
  the group $\langle x_tx_{2t-1},z\rangle$ preserves $\mathcal{B}$ and acts transitively upon it. In particular, $A_\mathcal{B} = A_H\langle x_tx_{2t-1},z\rangle$.
  Hence $\mathcal{B}$ is an orbit of a subgroup of $A$ which strictly contains the vertex-stabiliser $A_H$.  This shows that $\mathcal{B}$ is  a block of imprimitivity for $A$. 
  
  Now observe that each of the vertices $H$ and $Hz$ has a neighbour in each of the four translates $\mathcal{B}a$, $\mathcal{B}a^{-1}$, $\mathcal{B}x_{2t-1}a$, $\mathcal{B}x_ta^{-1}$ of the block $\mathcal{B}$ (this is obviously true for $H$, and follows
  easily for $Hz$ from the fact that $z$ is contained in the centre of $G$). On the other hand, a direct inspection shows that $Hx_{t}x_{2t-1}$ and $Hx_{t}x_{2t-1}z$ have no neighbours in these four translates of $\mathcal{B}$.
  In particular, the stabiliser $A_H$ cannot map the vertex $Hz$ to any of the other two vertices $Hx_{t}x_{2t-1}$ and $Hx_{t}x_{2t-1}z$ in the block $\mathcal{B}$. In particular, $A_H$ fixes the vertex $Hz$, and hence
  $A_H = A_{Hz}$.
  It follows that the group $\langle A_H, z \rangle$ preserves the set $\{ H, Hz\}$, acts upon it transitively, and contains the vertex-stabiliser $A_H$. This implies that its orbit $\{H, Hz\}$ is a block of imprimitivity for $A$, as claimed.
  
Now consider the kernel $K$ of the action of $A$ on the $\langle z \rangle$-orbits. Since $\Gamma \to \Gamma/\langle z \rangle$ is a covering projection, we know that $K$ acts semiregularly on $\V\Gamma$.
On the other hand $\langle z \rangle \le K$ has the same orbits on $\V\Gamma$ as $K$, and hence $K=K_v\langle z \rangle = \langle z \rangle$. In particular, $\langle z \rangle$ is normal in $A$. Therefore $A/\langle z \rangle \le \Aut(\Gamma/\langle z \rangle) \cong \Aut(\C(2t,t)) \cong G_{2t,t}$ (see Lemma~\ref{lem:AutCrs}). In particular, $|A| = 2 |G_{2t,t}| = |G|$, and therefore $A=G$.
\end{proof}

\begin{remark}\label{rk:1}{\rm
Since $G_t^+$ and $G_t^-$ are non-isomorphic groups, it follows  from Theorem~\ref{lemma:construction} that  $\Aut(\Gamma_t^+)\not\cong \Aut(\Gamma_t^-)$. In particular,
we obtain that $\Gamma_t^+$ and $\Gamma_t^-$ are non-isomorphic graphs. Moreover, since the girth of the graphs $\C(r,s)$ is $4$, none of the graphs $\Gamma_t^\pm$, other than possibly $\Gamma_2^+$,
is isomorphic to any of the graphs $\C(r,s)$. On the other hand, it can be easily checked that $\Gamma_2^+ \cong \C(4,3)$.
}
\end{remark}

\section{A family of $\Sym(n)$-arc-transitive graphs}\label{se:as}

In this section, we introduce another interesting family of tetravalent $G$-arc-transitive graphs with arbitrarily large vertex-stabilisers.
Unlike the graphs $\Gamma_t^\pm$, which have soluble automorphism groups, the graphs we are going to describe here have an almost simple arc-transitive group of automorphisms.

Let $m\geq 2$ be an integer and let $G$ be the symmetric group $\Sym(4m)$ acting on the set $\{1,2,\ldots, 4m\}$.
Define the following permutations of $G$
\begin{eqnarray*}
x_i & = &(2i-1,2i)\qquad \textrm{ for } 1\leq i\leq 2m-1, \\
h  &= & (4m-1,4m) \prod_{i=1}^{m-1} (2i-1,4m-2i-1)(2i,4m-2i)\\
a & = & (4m-2,4m) \prod_{i=1}^{m-1} (2i-1,4m-2i-3)(2i,4m-2i-2)\\
g  & = & (1,3,5,\ldots,4m-3)(2,4,6,\ldots,4m-4,4m-2,4m-1,4m).
\end{eqnarray*}

We can now define the graphs $\Delta_m$. 

\begin{definition}
\label{def:D}
For $G=\Sym(4m)$ and $H=\langle x_1, x_2, \ldots, x_{2m-1}, h \rangle$, let $\Delta_m = \Cos(G,H,a)$. 
\end{definition}

Before proving that the graphs $\Delta_m$ are indeed connected tetravalent graphs, we first observe that the following holds: 
\begin{equation}
\label{eq:9}
\begin{array}{lclc}
g & = & ah,\\
x_i^{h}  &=&x_{2m-i}           & \textrm{ for }1\leq i\leq 2m-1,\\
x_i^g    &=&x_{i+1}   &\textrm{ for } 1\leq i\leq 2m-2,\\
x_{2m-1}^g&=&(4m-3,4m-2)^g=(1,4m-1).& \\
\end{array}
\end{equation}
In particular, the group $H$ is isomorphic to a semidirect product $\C_2^{2m-1} \rtimes \C_2$ and has order $2^{2m}$.

\begin{theorem}
\label{lemma:sym}
For any $m\ge 2$, the graph $\Delta_m$ is  non-bipartite, connected, tetravalent and $G$-arc-transitive. Moreover,  $|\V\Delta_m|=(4m)!/2^{2m}$ and $|G_v| =2^{2m}$.
\end{theorem}

\begin{proof}
Let first prove that the triple $(G,H,a)$ from Definition~\ref{def:D} satisfies the conditions stated in Lemma~\ref{lem:cos}. In other words, let us prove that the core of $H$ in $G$ is trivial, that $G=\langle H,a\rangle$, and that $Ha^{-1}H = HaH$.
Since $a$ is an involution, the latter condition is automatically fulfilled. Furthermore, since the only nontrivial proper normal subgroup of $G=\Sym(4m)$ is the group $\Alt(4m)$
and since $H\not= \Alt(4m)$, it follows that the core of $H$ in $G$ is trivial. It remains to see that $G$ is generated by $H$ and $a$.

Set $K=\langle H,a\rangle$ and observe that from Equation~$(\ref{eq:9})$, $g\in K$ and hence $G$ is a transitive subgroup of $\Sym(4m)$. Furthermore, $g$ is a product of two cycles of lengths $2m-1$ and $2m+1$. Since $2m-1$ and $2m+1$ are coprime, it follows that $K$ is a primitive subgroup of $\Sym(4m)$. As $K$ contains the transposition $x_1=(1,2)$, we obtain from \cite[Theorem 3.3A(ii)]{DixMor} that $K=\Sym(4m)=G$. Lemma~\ref{lem:cos} now implies that $\Delta_m$ is a connected $G$-arc-transitive graph with $G_v \cong H$.

Finally, recall that $x_i^a=x_{2m-i-1}$ for $i=1,\ldots,2m-2$. In particular, $H\cap H^a\geq \langle x_1,\ldots,x_{2m-2}\rangle$ and $|H:H\cap H^a|$ divides $4$. Since $H\cap H^a$ is the stabiliser in $G$ of the arc $(H,Ha)$ and since $\Delta_m$ is $G$-arc-transitive, we obtain that the valency of $\Delta_m$ divides $4$. As $G$ is almost simple and acts faithfully on $\Delta_m$, we have that $\Delta_m$ has valency $4$. Since $H\nleq \Alt(4m)$, the group $\Alt(4m)$ is transitive on $\V\Delta_m$. As $\Alt(4m)$ is the only subgroup of index $2$ in $G$, we see that $\Delta_m$ is non-bipartite.
\end{proof}

Let $\Gamma=\Delta_m$ as in Definition~\ref{def:D} and let $G=\Sym(4m) \le \Aut(\Gamma)$. We consider the growth rate
 of the quantity $x=|G_v|=2^{2m}$ with respect to $y=|\V\Gamma| = (4m)!/2^{2m}$. Using Stirling's formula for the factorial term in $y$, one easily obtains that
\begin{equation*}
 y \approx \sqrt{4\pi\log(x)} \,x^{1-2\log(e)}\, \log(x)^{\log(x^2)},
\end{equation*}
with all the logarithms having base $2$. Since $\log(x)^{\log(x^2)} = x^{2\log(x)\log(\log(x))}$, this shows that  $x=|G_v|$ grows faster than any logarithmic function on $y=|V\Gamma|$ but slower than $|\V\Gamma|^c$ for any $c>0$. The graphs $\Delta_m$ are thus quite far from attaining the bound given in \eqref{bound}, yet we conjecture that a much faster growth of $|G_v|$ cannot be expected when $G$ is  almost simple.

\begin{conjecture}
For any positive constant $c$ there exists a finite family of graphs $\F_c$, such that the following holds:
if $G$ is an almost simple group and $\Gamma$ is a connected tetravalent $G$-arc-transitive graph such that $\Gamma$ is not contained in $\F_c$, then $|G_v|<|\V\Gamma|^c$.
\end{conjecture}


\thebibliography{15}

\bibitem{magma} W.~Bosma, J.~Cannon and C.~Playoust, The Magma algebra system. I. The user language, \textit{J. Symbolic Comput.}, \textbf{24}(3-4), (1997), 235--265. 

\bibitem{CW} M.~D.~E. Conder and C.\ G.\ Walker, Vertex-Transitive Non-Cayley Graphs with Arbitrarily Large Vertex-Stabilizer, {\em J.\  Alg.\ Combin.} {\bf 8} (1998), 29--38.

\bibitem{DixMor}J.~D.~Dixon and B.~Mortimer, Permutation Groups, Springer-Verlag, New York, (1996).

\bibitem{main} P.~Poto\v{c}nik, P.~Spiga and G.~Verret,  \textit{Bounding the order of the vertex-stabiliser in $3$-valent vertex-transitive and $4$-valent arc-transitive graphs}, submitted.

\bibitem{PraegerXu} C.~E.~Praeger and M.~Y.~Xu, A Characterization of a Class of Symmetric Graphs of Twice Prime Valency, \textit{Europ.\ J.\ Combin.} \textbf{10} (1989), 91--102.

\bibitem{Sabidussi} G.~Sabidussi, Vertex-transitive graphs, \textit{Monatshefte Math.} \textbf{68} (1961), 426--438.

\bibitem{Tutte} W.~T.~Tutte, A family of cubical graphs, \textit{Proc. Camb. Phil. Soc} \textbf{43} (1947), 459--474.

\bibitem{Tutte2} W.~T.~Tutte, On the symmetry of cubic graphs, \textit{Canad. J. Math.} \textbf{11} (1959), 621--624.
\end{document}